\documentclass[a4paper]{article}

\usepackage{ifthen}
\usepackage{amsmath,amsfonts,amssymb}
\usepackage{tabularx,supertabular,booktabs}
\usepackage{array,multirow,graphicx,bigdelim}
\usepackage{todonotes}
\usepackage{xspace}
\usepackage{pdflscape}
\usepackage{natbib}
\usepackage{lscape}
\usepackage{subfig}
\usepackage {footnote}
\usepackage{footmisc}
\usepackage{mathtools}
\usepackage {tikz}
\usepackage{zibtitlepage}
\usepackage{amsthm}
\usepackage[linesnumbered,ruled]{algorithm2e}
\usepackage{natbib}
\setcitestyle{numbers}
\usetikzlibrary{shapes,arrows, matrix,decorations}
\usetikzlibrary{decorations.pathreplacing}
\usetikzlibrary{automata}

\usepackage{siunitx}

\usepackage{url}
\usepackage[bookmarks]{hyperref}

\newcommand{\myurl}[1]{\textsf{\footnotesize \url{#1}}\xspace}%
\newcommand{\mc}{{MaxCut}\xspace}
\newcommand{\qubo}{{QUBO}\xspace}

\newcommand{\with}{\ensuremath{\mid}\xspace}

\newcommand{\solver}[1]{\textsc{#1}\xspace}
\newcommand{\scip}{\solver{SCIP}}
\newcommand{\biqbin}{\solver{BiqBin}}
\newcommand{\biqmac}{\solver{BiqMac}}
\newcommand{\biqcrunch}{\solver{BiqCrunch}}
\newcommand{\madam}{\solver{MADAM}}
\newcommand{\mcsparse}{\solver{McSparse}}
\newcommand{\soplex}{\solver{SoPlex}}
\newcommand{\cplex}{\solver{CPLEX}}
\newcommand{\UG}{\solver{UG}}
\newcommand{\gurobi}{\solver{Gurobi}}

\newcommand{\qplib}{\solver{QPLIB}}

\newtheorem{proposition}{Proposition}

\newtheorem{formulation}{Formulation}

\ZTPAuthor{ \ZTPHasOrcid{Daniel Rehfeldt}{0000-0002-2877-074X},
  \ZTPHasOrcid{Thorsten Koch}{0000-0002-1967-0077},
   \ZTPHasOrcid{Yuji Shinano}{0000-0002-2902-882X}}
\ZTPTitle{Faster exact solution  of sparse MaxCut and QUBO problems}
\ZTPNumber{22-02}
\ZTPMonth{February}
\ZTPYear{2022}

\title{Faster exact solution  of sparse MaxCut and QUBO problems}
\author{\ZTPHasOrcid{Daniel Rehfeldt}{0000-0002-2877-074X}\and\
  \ZTPHasOrcid{Thorsten Koch}{0000-0002-1967-0077}\and\
      \ZTPHasOrcid{Yuji Shinano}{0000-0002-2902-882X}}
\hypersetup{pdftitle={Exact solution of MaxCut and QUBO},
  pdfauthor={Daniel Rehfeldt, Thorsten Koch}}

\begin{document}
	
\maketitle

	\begin{abstract}
		The maximum-cut problem is one of the fundamental problems in combinatorial optimization. 
	With the advent of quantum computers, both the maximum-cut and the equivalent quadratic unconstrained binary optimization problem have experienced much interest in recent years.
	
	This article aims to advance the state of the art in the exact solution of both problems---by using mathematical programming techniques on digital computers. The main focus lies on sparse problem instances, although also dense ones can be solved.
	We enhance several algorithmic components such as reduction techniques and cutting-plane separation algorithms, and combine them in an exact branch-and-cut solver. Furthermore, we provide a parallel implementation.
	The new solver is shown to significantly outperform existing state-of-the-art software for sparse \mc and \qubo instances.
		Furthermore, we improve the best known bounds for several instances from the 7th DIMACS Challenge and the QPLIB, and solve some of them (for the first time) to optimality.

	\end{abstract}
		\section{Introduction}
		\label{sec:introduction}
		Given an undirected graph $G=(V,E)$, and edge weights $w: E \rightarrow
		\mathbb{Q}$, the \emph{maximum-cut} (\mc) problem is to find a partition $(V_1,V_2)$ of $V$ such that the summed weight of the edges between $V_1$ and $V_2$ is maximized.
		\mc is one of the fundamental $\mathcal{NP}$-hard optimization problems~\cite{Karp72} and has applications for example in VLSI design~\cite{Barahona88} and the theory of spin glasses in physics~\cite{Liers2004}\footnote{As a side note, the 2021 Nobel prize in Physics was awarded for work on spin glasses.}. The latter application is particularly interesting, because it 
		requires an exact solution of the \mc problem.
		
		A  problem that is equivalent to \mc is the \emph{quadratic unconstrained binary optimization} (\qubo) problem. Given a matrix $Q \in \mathbb{Q}^{n \times n}$, the corresponding \qubo problem can be formulated as
		  \begin{align*}
		    \min \;  x^T Q x & ~~~~~ \\
		     x \in \{0,1\}^n. &
		  \end{align*}
		  Any \qubo instance can be formulated as a \mc instance in a graph with $n+1$ vertices, and any \mc instance on a graph $(V,E)$ can be formulated as a \qubo instance with $n = |V|-1$, see e.g.~\cite{Barahona89}. The focus of this article is mostly on \mc algorithms, but due to the just mentioned equivalence, all results can be (and indeed are) applied to \qubo as well.
		  
The huge recent interest in quantum computing has also put \mc and \qubo in the spotlight: Both of them can be heuristically solved by current quantum annealers. However, J\"unger et al.~\cite{Juenger2021} demonstrate on a wide range of test-sets that digital computing methods prevail against state-of-the-art quantum annealers.

For digital computers, many heuristics have been proposed both for \mc and \qubo. See Dunning et. al.~\cite{Dunning18} for a recent overview. There have also been various articles on exact solution. See Barahona et al.~\cite{Barahona89} for an early, Rendl et al.~\cite{Rendl10} for a more recent, and J\"unger et al.~\cite{Juenger2021} for an up-to-date overview.
In the last years, more focus has been put on the development of methods that are best suited for \emph{dense} instances, see for example~\cite{Hrga19,Hrga21,Krislock17} for state-of-the-art methods. However, the maximum number of nodes for \mc (or number of variables for \qubo)  instances that can be handled by these methods is roughly 300.
In contrast, this article aims to advance the state of the art in the practical exact solution of \emph{sparse} \mc and \qubo instances. The largest (sparse) instance solved in this article has more than 10\,000 nodes.	
	
		\subsection{Contribution and structure}
		
		This article describes the design and implementation of a branch-and-cut based \mc and \qubo solver.
		In particular, we suggest several algorithmic improvements of key components of a branch-and-cut framework.
		
		Section~\ref{sec:relaxation} shows how to efficiently solve a well-known linear programming (LP) relaxation for the \mc problem by using cutting planes. Among other things, we demonstrate how the separation of maximally violated constraints, which was described by many authors as being too slow for practical use, can be realized with quite moderate run times.

	Section~\ref{sec:reduction}	is concerned with another vital component within branch-and-cut: reduction techniques.
	 We review methods from the literature and propose new ones. The reduction methods can be applied for preprocessing and domain propagation.
		
	Section~\ref{sec:branch-and-cut} shows how to integrate the techniques from the previous to sections as well as several additional methods in a branch-and-cut algorithm. Parallelization is also discussed.
	
	Section~\ref{sec:computational} provides computational results of the newly implemented \mc and \qubo solver on a large collection of test-sets from the literature.
	It is shown that the new solver outperforms the previous state of the art. Furthermore, the best known solutions of several benchmark instances can be improved and one is even solved (for the first time) to optimality.
		
		\subsection{Preliminaries and notation}
 
In the remainder of this article, we assume that a \mc instance $I_{MC} = (G,w)$ with graph $G=(V,E)$ and edge weights $w$ is given.	 
Graphs are always assumed to be undirected and simple, i.e., without parallel edges or self-loops. Given a graph $G = (V,E)$, we refer to the vertices and edges of any subgraph $G' \subseteq G$ as $V(G')$ and $E(G')$ respectively,
 An edge between vertices $u, v \in V$ is denoted by $\{u,v\}$. An edge set $C = \left \{ \{v_1,v_2\},\{v_2,v_3\},...,\{v_{k-1} ,v_k\}  \right \}$ is called a \emph{cycle}. A cycle $C$ is called \emph{simple} if all its vertices have degree $2$ in $C$.
 An edge $\{u,w\} \in E \setminus C$ is called a \emph{chord} of $C$ if both $u$ and $w$ are contained in (an edge of) $C$. If no such $\{u,w\}$ exists, we say that $C$ is \emph{chordless}.
 Given a graph  $G = (V,E)$ and a $U \subseteq V$, we define the induced \emph{edge cut} as $\delta(U):=\{ \{u,v\} \in E \with u\in U, v\in V\setminus U\}$. 

Finally, for any function $x: M \mapsto \mathbb{R}$ with $M$ finite, and any $M' \subseteq M$, we define $x(M') := \sum_{i \in M'} x(i)$.

\section{Solving the relaxation: efficient separation of odd-cycle cuts}
\label{sec:relaxation}

This section is concerned with an integer programming (IP) formulation for \mc due to Barahona and Mahjoub~\cite{Barahona86}, given below.

\begin{formulation}{Odd-cycle cuts}
  \label{form:dcut}
  \begin{eqnarray}
   \textrm{max}\quad {w}^T x\\
    \label{form:dcut:1}
    \text{s.t. } \sum_{e \in F} x(e) - \sum_{e \in C \setminus F} x(e)  &\leq& |F| - 1 \hspace{1mm}\text{for all cycles } C, F \subseteq C, |F| \text{ odd} \\
    \label{form:dcut:2}
    x(e)&\in& \{0,1\} \hspace{3.1mm}\text{for all } e \in E.
  \end{eqnarray}
\end{formulation}
  The formulation is based on the observation that for any edge cut $\delta(U)$ and any cycle $C$ the number of their common edges, namely $|C \cap \delta(U)|$, is even. 
  This property is enforced by the constraints~\eqref{form:dcut:1}. These constraints are called \emph{cycle inequalities}.

\subsection{Cutting plane separation}

Barahona and Mahjoub~\cite{Barahona86} show that the LP-relaxation of Formulation~\ref{form:dcut} can be solved in polynomial time.
More precisely, they describe how to separate the constraints~\eqref{form:dcut:2} in polynomial time, as demonstrated in the following.
First, rewrite the constraints~\eqref{form:dcut:2} as
\begin{equation}
\label{form:dcut:1b}
\sum_{e \in F} \left(1-x(e) \right) + \sum_{e \in C \setminus F} x(e)  \ge 1 \hspace{3mm}\text{for all cycles } C, F \subseteq C, |F| \text{ odd}.
\end{equation}
Next, construct a new graph $H$ from the \mc graph $G = (V,E)$. This graph $H$ consists of two copies $G' = (V',E')$ and $G'' = (V'',E'')$ of $G$, connected by the following additional edges. For each $v \in V$ let $v'$ and $v''$ be the corresponding vertices in $G'$ and $G''$, respectively. For each edge $\{v,w\} \in E$ let $\{v',w''\}$ 
and  $\{v'',w'\}$ be in $H$. Finally, for any (LP-relaxation vector) $x \in [0,1]^E$ define the following edge weights $p$ on $H$:
For each $e = \{v,w\} \in E$, set $p(\{v',w'\}) := p(\{v'',w''\}) := x(e)$ and  $p(\{v',w''\}) := p(\{v'',w'\}) := 1 - x(e)$. The construction is exemplified in Figure~\ref{fig:cutgraph}. Consider, for example, the edge $\{v,w\}$ in Figure~\ref{fig:graph_org}. The weight $p$ of the corresponding (dashed) edges $\{v',w''\}$ and $\{v'',w'\}$ in Figure~\ref{fig:graph_sepa}  is $1 - x(\{v,w\})$. The weight $p$ of the corresponding (bold) edges $\{v',w'\}$ and $\{v'',w''\}$ is $x(\{v,w\})$.

	\begin{figure}[ht]
  		\centering
  		\subfloat[Original graph.]{
  		\begin{tikzpicture}[scale=1.2]
  		  \tikzstyle{steiner} = [circle, draw, thick, minimum size=0.2, inner sep=3.0pt];
  		  \tikzstyle{invisible} = [circle, draw, thick, minimum size=0.0, inner sep=0.0pt];
  		  \node[steiner, label=left:{\small  ~}] (u) at (0,0.0) {\small $u$};
  		  \node[steiner, label=left:{\small  ~}] (v) at (0,3.0) {\small $v$};
  		  \node[steiner, label=left:{\small  ~}] (w) at (0.75,1.5) {\small $w$};
  		
  		  \node[invisible, label=left:{\small  ~}] (x) at (1.5,1.5) {};

  		  \draw[] (u) -- (v) node [midway, left=-0.0pt] {\small };
  		  \draw[] (v) -- (w) node [midway, above right=-0.0pt] {\small };
  		  \draw[] (w) -- (u) node [midway, above right=-0.0pt] {\small };
  		  
  		    \node[invisible, label=left:{\small  ~}] (x) at (-1.2,-1.2) {};
  		
  		\end{tikzpicture}
  			\label{fig:graph_org}
  		}
  		\hfill
  		\subfloat[Auxiliary graph, consisting of two copies of the original graph (bold edges), and additional connecting edges (dashed). ]{
  		\begin{tikzpicture}[scale=1.2]
  		  \tikzstyle{steiner} = [circle, draw, thick, minimum size=0.2, inner sep=1.5pt];
  		  \tikzstyle{invisible} = [circle, draw, thick, minimum size=0.0, inner sep=0.0pt];
  		  \node[steiner, label=left:{\small  ~}] (u1) at (0,0.0) {\small $u'$};
  		  \node[steiner, label=left:{\small  ~}] (v1) at (0,3.0) {\small $v'$};
  		  \node[steiner, label=left:{\small  ~}] (w1) at (0.75,1.5) {\small $w'$};
  		
  		  \node[steiner, label=left:{\small  ~}] (u2) at (3,0.0) {\small $u''$};
  		  \node[steiner, label=left:{\small  ~}] (v2) at (3,3.0) {\small $v''$};
  		  \node[steiner, label=left:{\small  ~}] (w2) at (2.25,1.5) {\small $w''$};
  		
  		  \draw[] (u1) -- (v1) node [midway, left=-0.0pt] {\small };
  		  \draw[] (v1) -- (w1) node [midway, above right=-0.0pt] {\small };
  		  \draw[] (w1) -- (u1) node [midway, above right=-0.0pt] {\small };
  		  
  		  \draw[] (u2) -- (v2) node [midway, left=-0.0pt] {\small };
  		  \draw[] (v2) -- (w2) node [midway, above right=-0.0pt] {\small };
  		  \draw[] (w2) -- (u2) node [midway, above right=-0.0pt] {\small };

  		  \draw [dashed]   (u1) to[out=-20,in=-70, distance=3.5cm ] (v2);
  		  \draw [dashed]   (v1) to[out=20,in=70, distance=3.5cm ] (u2);

  		  \draw[dashed] (v1) -- (w2) node [midway, left=-0.0pt] {\small };
  		  \draw[dashed] (w1) -- (v2) node [midway, left=-0.0pt] {\small };
  		  
  		  \draw[dashed] (u1) -- (w2) node [midway, left=-0.0pt] {\small };
  		  \draw[dashed] (w1) -- (u2) node [midway, left=-0.0pt] {\small };
  		
  		\end{tikzpicture}
  			\label{fig:graph_sepa}
  		}
	\caption{\mc graph and corresponding auxiliary graph for cycle cut separation. } 
  		\label{fig:cutgraph}
  	\end{figure}
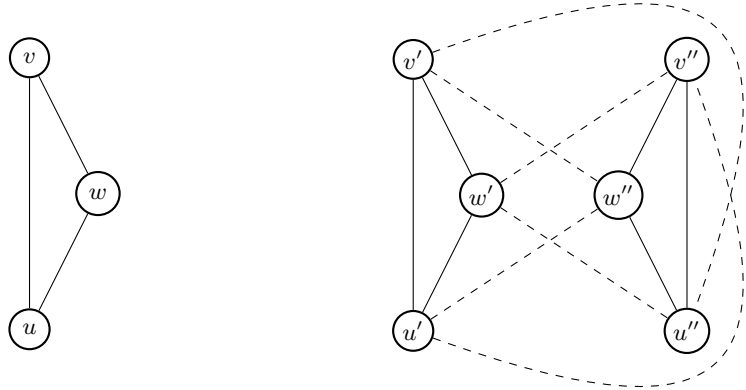

Given an LP-relaxation vector $x \in [0,1]^E$, we can find violated inequalities~\eqref{form:dcut:2} as follows.
For each $v \in V$ compute a shortest path between $v'$ and $v''$ in the weighted graph $(H,p)$. 
By construction of $H$, such a path contains an odd number of edges which are neither in $E'$ nor $E''$. Let $F$ be the corresponding set of edges in $E$; i.e. for each edge $\{v',w''\}$ or $\{v'',w'\}$ that is in the shortest path, let $\{v,w\}$ be in $F$.
Furthermore, the edges of the shortest path correspond to a closed walk $C$ in $G$.
The length of the shortest path in $(H,p)$ is equal to $\sum_{e \in F} \left(1-x(e) \right) + \sum_{e \in C \setminus F} x(e)$.
Thus, if for each $v \in V$ the corresponding shortest path between $v'$ and $v''$ in $(H,p)$ has length at least $1$, the vector $x$ is an optimal solution to the LP-relaxation of Formulation~\ref{form:dcut}. Otherwise, we have found at least one violated constraint.

Although shortest paths can be computed in polynomial time, the literature has so far considered the above separation procedure as too time-consuming to be directly used in practical exact \mc or \qubo solution.
Instead, heuristics are employed and exact cycle separation is only used if no more cuts can be found otherwise, see, e.g.,~\cite{Barahona88,Barahona89,Bonato2014,Juenger2021,Liers2004}.
However, as we will show in the following, the exact separation can actually be realized in a practically quite efficient way.

\subsection{Fast computation of maximally-violated constraints}
\label{sec:relaxation:fast}

Initially, we observe that it is usually possible to considerably reduce the size of the auxiliary graph $H$ described above.
First, all edges $e$ of $H$ with $p(e) = 1$ (or practically, with $p(e)$ being sufficiently close to $1$) can be removed. Because all edge weights are non-negative, no such edges can be contained in a path of weight smaller than $1$.
 Second, one can contract edges $e$ with $p(e) = 0$.
Both of these operations can be done implicitly while creating the auxiliary graph (e.g., edges with weight $1$ are never added). In this way, one can use cache-efficient, static data structures, such as the compressed-sparse-row format, see e.g.~\cite{kepner2011}, for representing the auxiliary graph.

For computing a shortest path, we use a modified version of Dijkstra's algorithm. For any vertex $v$ in the auxiliary graph let $d(v)$ denote the distance of $v$ to the start vertex, as computed by the algorithm. We use the following modifications.
 First, we stop the execution of the algorithm as soon as we scan a vertex $v$ with $d(v) \geq 1$. Second, as already observed in J\"unger and Mallach~\cite{Junger19}, one does need to not proceed the shortest path computation from any vertex, say $v'$, in the auxiliary graph where the twin vertex,  $v''$, has already been scanned and the following condition holds: $d(v') + d(v'') \geq 1$.

Finally, we use an optimized implementation of Dijkstra's algorithm together with a specialized binary heap. For the latter, we exploit the fact that the values (i.e. vertex indices) of the key, value pairs inserted into the heap are natural numbers bounded by the number of vertices of the auxiliary graph.

\subsection{Post-processing}
\label{sec:relaxation:post}

As already mentioned above, the edges of the shortest path computed in the auxiliary graph correspond to a closed walk $G$---but not necessarily to a simple cycle.
Thus, J\"unger and Mallach~\cite{Junger19} suggest to extract all simple cycles from such a closed walk and separate the corresponding inequalities. We follow this suggestion (although we note that this modification is performance neutral in our implementation).

Barahona and Mahjoub~\cite{Barahona86} observe that a cycle inequality is only facet-defining if the corresponding cycle is chordless.
If a cycle $C$ has a chord $e$, one readily obtains two smaller cycles $C_1$ and $C_2$ with $C_1 \cup C_2 = C \cup \{e\}$ and $C_1 \cap C_2 = \{e\}$.
One verifies that any cycle inequality defined on $C$
can be written as the sum of two cycle inequalities defined on $C_1$ and $C_2$, where $e$ is in the odd edges set $F$ of exactly one of the two cycle inequalities.
J\"unger and Mallach~\cite{Junger21informs} suggest a procedure to extract from any simple cycle $C$ with corresponding violated cycle-inequality a chordless cycle $C'$ whose cycle-inequality is also violated. This procedure runs in $O(|E|)$.
However, a disadvantage of this approach is that it finds only one such chordless cycle, which might not be the smallest or most violated one. Additionally, there can be several such chordless cycles. In the following, we suggest a procedure to find several non-overlapping 
chordless cycles with corresponding violated cycle inequality from a given cycle $C$ with violated cycle inequality.

Consider a simple cycle $C = \left \{ \{v_1,v_2\},\{v_2,v_3\},...,\{v_{k-1} ,v_k\}  \right \}$ and let $F \subseteq C$ with $|F|$ odd. Assume there is a vector $x \in [0,1]^E$ such that the cycle inequality corresponding to $C$ and $F$ is violated, that is:
\begin{equation*}
\sum_{e \in F} \left(1-x(e) \right) + \sum_{e \in C \setminus F} x(e)  < 1.
\end{equation*}
For each $i = 2,...,k$ define $P_i := \left\{ \{v_1,v_2\}, \{v_2,v_3\},...,\{v_{i-1},v_i\} \right\}$ and store the following information.
\begin{itemize}
\item $f(i) := |F \cap P_i|$,
\item $q(i) := \sum_{e \in F \cap P_i} \left(1-x(e) \right) + \sum_{e \in (C \cap P_i) \setminus F} x(e)$.
\end{itemize}        
This information can be computed in total time $O(|C|)$: Traverse the nodes $v_i$, $i = 2,3,..,k$  of $C$ in this order and compute the above two values for $i$ from $i-1$.

With the above information at hand, traverse for each $i = 2,...,k$ the incident edges of $v_i$. Whenever a chord $\{v_i,v_j\}$ with $j<i$ is found, check whether the cycle inequality of one or both of the corresponding cycles is violated. This check can be performed in constant time by
using the precomputed information for the indices $i$ and $j$. For example, if $f(i) - f(j)$ is even, one of the corresponding two cycle inequalities is
\begin{equation*}
q(i) - q(j) + 1 - x(\{v_i,v_j\}) \geq 1.
\end{equation*} 

If a violated cycle inequality is found, add the corresponding chord together with a flag that indicates which of the two possible cycles is to be used to some (initially empty) queue $R$.
Once the incident edges of all nodes $v_i$ for  $i = 2,...,k$ have been traversed,  
sort the elements of $R$ according to the size of the corresponding cycles in non-decreasing order. Consider all indices of the original cycle as unmarked.
Check the (implicit) cycles in $R$ in non-decreasing order. Let $\{v_i,v_j\}$ with $i<j$ be the corresponding chord. If both $i$ and $j$ are unmarked, mark the indices $i+1,i+2,...,j-1$. Otherwise, discard the (implicit) cycle. Finally, add all cycle inequalities corresponding to non-discarded cycles to the cut pool.
The overall procedure runs in $O(|E| \log(|E|))$. In practice, its run time is completely neglectable.

Finally, we suggest a procedure to obtain additional cycle cuts from the auxiliary graph.
 This approach is particularly useful for \mc instances with few vertices, because the number of generated cycle inequalities separated in each round is limited by the number of vertices of the \mc instance (if we ignore additional cycle-inequalities that are possibly found by the above post-processing).
 The procedure makes use of the symmetry of the auxiliary graph. Assume that we have computed a shortest path between a pair of vertices, say $v'$ and $v''$, as described above. Recall that $d(w)$ denotes the distance of any vertex $w$ to the start vertex $v'$.
 If there is a twin pair of vertices $u',u''$ such that none of them are part of the shortest path between $v'$ and $v''$, and $d(u') + d(u'') < 1$, we can get another violated cycle inequality as follow:
First, we take the $v'$-$u'$ path computed by the algorithm. Second, we consider the $v'$-$u''$ path, and transform it to an $u''$-$v''$ path (of same length) by exploiting the symmetry of the auxiliary graph. By combining the two paths, we obtain an $v'$-$v''$ path of length $d(u') + d(u'')$.

\section{Simplifying the problem: reduction techniques}
\label{sec:reduction}
Reduction techniques are a key ingredient for the exact solution of many $\mathcal{NP}$-hard optimization problems, such as Steiner tree~\cite{RehfeldtKoch21_ipco} or vertex coloring~\cite{LinCLS17}.
For \qubo, several reductions methods have been suggested in the literature. Basic techniques can already be found in Hammer et al.~\cite{hammer68}. The perhaps most extensive reduction framework is given in Tavares et. al.~\cite{boros2006preprocessing}. Recently, Glover et al.~\cite{Glover18} provided efficient realizations and extensions of the classic \emph{first} and \emph{second order} \emph{derivative} and \emph{co-derivative} techniques~\cite{hammer84}. We have implemented the methods from Glover et al.~\cite{Glover18} for this article. 
However, we do not provide details, but rather concentrate on \mc reduction techniques in the following.  

For \mc, there are several articles that discuss reduction techniques for unweighted \mc. Ferizovic et al.~\cite{Ferizovic20} provide the practically most powerful collection of such techniques.
Lange et al.~\cite{Lange19} provide techniques for general (weighted) \mc instances. In the following, we will describe some of their methods. Furthermore, we suggest new \mc reduction methods. Their practical strength will be demonstrated in Section~\ref{sec:computational}.

Initially, we note that any edge with weight $0$ can be removed from $I_{MC}$. Any solution to this reduced version of $I_{MC}$ can be extended to a solution of same weight to the original instance (in linear time).
Thus, in the following we assume no edges have weight $0$.
We also note that for the incidence vector $x \in \{0,1\}^E$ of any graph cut one obtains a  corresponding (but not unique) vertex assignment $y \in \{0,1\}^V$ that satisfies for all $\{u,v\} \in E$ the relation $y(u) \neq y(v) \iff x(\{u,v\}) = 1$. This correspondence will be used repeatedly in the following.

\subsection{Cut-based reduction techniques}
\label{sec:reduction:cut}
The first reduction technique from Lange et al.~\cite{Lange19} is based on the following proposition.

\begin{proposition}[\cite{Lange19}]
\label{prop:rededge}
Let $e \in E$ and $U \subset V$ such that $e \in \delta(U)$.
If
\begin{equation*}
|w(e)| \geq \sum_{a \in \delta(U) \setminus \{e\}} |w(a)|,
\end{equation*}
then there is an optimal solution $x \in \{0,1\}^E$ to $I_{MC}$ with $x(e) = \beta$, where $\beta = 1$ if $w(e) > 0$, and $\beta = 0$ if $w(e) < 0$.
\end{proposition}
Note that in the case of $x(e) =0$, one can simply contract $e$. In the case of $x(e) =1$, one needs to multiply the weights of the incident edges of one of the endpoints of $e$ by $-1$ before the contraction.

One way to check for all $e \in E$ whether an $U \subset V$ exists such that the conditions of Proposition~\ref{prop:rededge} are satisfied is by using Gomory-Hu trees.
We have only implemented a simpler check that considers for an edge $e = \{v,u\} \in E$ the sets $\{v\}$ and $\{u\}$ as $U$, as already suggested in Lange et al.~\cite{Lange19}.
A combined check for all edges can be made in $O(|E|)$. We note that this test corresponds to the first order derivative reduction method (mentioned above) for \qubo. This relation can be readily verified by means of the standard transformations between \mc and \qubo.

The next reduction technique from Lange et al.~\cite{Lange19} is based on triangles, and is given below.
\begin{proposition}[\cite{Lange19}]
\label{prop:redtriag}
Assume there is a triangle in $G$ with edges $\{v_1,v_2\}$, $\{v_1,v_3\}$, $\{v_2,v_3\}$. Let $V_1 \subset V$ such that $\{v_1,v_2\}, \{v_1,v_3\} \subset \delta(V_1)$, and $V_2 \subset V$ such that $\{v_1,v_2\}, \{v_2,v_3\} \subset \delta(V_2)$.
If
\begin{equation*}
w(\{v_1,v_3\}) + w(\{v_1,v_2\})   \geq \sum_{e \in \delta(V_1) \setminus \left\{ \{v_1,v_3\}, \{v_1,v_2\}\right\}} |w(e)|
\end{equation*}
and
\begin{equation*}
w(\{v_1,v_2\}) + w(\{v_2,v_3\})   \geq \sum_{e \in \delta(V_2) \setminus \left\{ \{v_1,v_2\}, \{v_2,v_3\}\right\}} |w(e)|,
\end{equation*}
then there is an optimal solution $x \in \{0,1\}^E$ to $I_{MC}$ with $x(\{v_1,v_2\}) = 0$.
\end{proposition}
Similarly to the previous proposition, we only implemented tests for the simple cases of  $\{v_1\}$, $\{v_2\}$, $\{v_1,v_3\}$, and $\{v_2,v_3\}$ for $V_1$ and $V_2$, respectively.

In the following, we propose a new reduction test based on triangles, which complements the above one from Lange et al.~\cite{Lange19}.

\begin{proposition}
\label{prop:redtriag2}
Assume there is a triangle in $G$ with edges $\{v_1,v_2\}$, $\{v_1,v_3\}$, $\{v_2,v_3\} \in E$ such that 
$w(\{v_1,v_2\}) > 0$, $w(\{v_1,v_3\}) > 0$, and $w(\{v_2,v_3\}) < 0$. Let $V_1 \subset V$ such that $\{v_1,v_2\}, \{v_1,v_3\} \in \delta(V_1)$ and let $V_2 \subset V$ such that $\{v_1,v_2\}, \{v_2,v_3\} \in \delta(V_2)$.
If 
\begin{equation}
	 \label{prop:trinew:1}
w(\{v_1,v_2\}) +  w(\{v_1,v_3\}) \geq \sum_{e \in \delta(V_1) \setminus \left \{  \{v_1,v_2\},  \{v_1,v_3\} \right \} } |w(e)|,
\end{equation}
and
\begin{equation}
	\label{prop:trinew:2}
w(\{v_1,v_2\}) -  w(\{v_2,v_3\}) \geq \sum_{e \in  \delta(V_2) \setminus \left \{  \{v_1,v_2\},  \{v_1,v_3\} \right \} } |w(e)|,
\end{equation}
then there is an optimal solution $x \in \{0,1\}^E$ to $I_{MC}$ such that $x(\{v_1,v_2\}) = 1$.
\end{proposition}

\begin{proof}

Let	 $x \in \{0,1\}^E$ be a feasible solution to  $I_{MC}$ with $x(\{v_1,v_2\}) = 0$. We will construct a feasible solution $x'  \in \{0,1\}^E$ with $x(\{v_1,v_2\}) = 1$ such that $w^T x' \geq w^T x$. Thus, there exists at least one optimal solution  $x \in \{0,1\}^E$ with  $x(\{v_1,v_2\}) = 1$.

Because $x(\{v_1,v_2\}) = 0$, it needs to holds that either
\begin{equation}
	 \label{prop:trinew:3}
x(\{v_1,v_3\}) =  x(\{v_2,v_3\}) = 0 
\end{equation}
or
\begin{equation}
x(\{v_1,v_3\}) =  x(\{v_2,v_3\}) = 1.
\end{equation}
We just consider the case~\eqref{prop:trinew:3}; the second one can be handled in an analogeous way.
Let $y \in \{0,1\}^V$ be a vertex assignment corresponding to $x$; i.e., for all $\{u,v\} \in E$ it holds that $y(u) \neq y(v) \iff x(\{u,v\}) = 1$.
Define a new vertex assignment  $y' \in \{0,1\}^V$ as follows
\begin{equation*}
y'(v) := \left\{
\begin{array}{ll}
1 - y(v) & \textrm{if } v \in V_1 \\
y(v) & \, \textrm{otherwise.} \\
\end{array}
\right. 
\end{equation*}
Let $x' \in \{0,1\}^E$ be the cut corresponding to $y'$; i.e.,  for all $\{u,v\} \in E$ it holds
 $x'(\{u,v\}) = 1$ if $y(u) \neq y(v)$, and $x'(\{u,v\}) = 0$ otherwise. Note that for all $e \in E \setminus \delta(V_1)$ it holds that $x'(e) = x(e)$. For all $e \in \delta(V_1)$ it holds that 
 $x'(e) = 1 - x(e)$. In particular, 
 \begin{equation}
 \label{prop:trinew:4}
 x'(\{v_1,v_2\}) =  x'(\{v_1,v_3\}) = 1,
 \end{equation}
 because of $x(\{v_1,v_2\}) =  x(\{v_1,v_3\}) = 0$.
  Thus, we obtain
 \begin{align*}
 \sum_{e \in E} w(e) x'(e) &=  \sum_{e \in E \setminus \delta(V_1) } w(e) x'(e) + \sum_{e \in \delta(V_1) } w(e) x'(e)  \\
  &= \sum_{e \in E \setminus \delta(V_1) } w(e) x(e) + \sum_{e \in \delta(V_1) } w(e)  x'(e) \\
   &\stackrel{\eqref{prop:trinew:4}}{=} \sum_{e \in E \setminus \delta(V_1) } w(e) x(e) \\
   & + \sum_{e \in \delta(V_1) \setminus \left \{  \{v_1,v_2\},  \{v_1,v_3\} \right \}  } w(e) x'(e) + w(\{v_1,v_2\})+ w(\{v_1,v_3\})  \\
  &\stackrel{\eqref{prop:trinew:1}}{\geq} \sum_{e \in E \setminus \delta(V_1) } w(e) x(e) \\
  & + \sum_{e \in \delta(V_1) \setminus \left \{  \{v_1,v_2\},  \{v_1,v_3\} \right \}  } (w(e) x'(e) +  |w(e)|) \\
    &\geq \sum_{e \in E \setminus \delta(V_1) } w(e) x(e) \\
    & + \sum_{e \in \delta(V_1) \setminus \left \{  \{v_1,v_2\},  \{v_1,v_3\} \right \}  } w(e) (1- x'(e)) \\
    &= \sum_{e \in E \setminus \delta(V_1) } w(e) x(e) \\
    & + \sum_{e \in \delta(V_1) \setminus \left \{  \{v_1,v_2\},  \{v_1,v_3\} \right \}  } w(e) x(e) \\
  &= \sum_{e \in E \setminus \delta(V_1) } w(e) x(e)  \sum_{e \in \delta(V_1) } w(e) x(e)   \\
  &=  \sum_{e \in E} w(e) x(e),
 \end{align*}
 which concludes the proof.
\end{proof}

As for the previous triangle test, we only consider the simple cases of  $\{v_1\}$, $\{v_2\}$, $\{v_1,v_3\}$, and $\{v_2,v_3\}$ for  $V_1$ and $V_2$ in our implementation.

Note that Lange et al.~\cite{Lange19} furthermore propose a generalization of Proposition~\ref{prop:redtriag} to more general connected subgraphs. Also Proposition~\ref{prop:redtriag2} could be generalized in a similar way.
However, since we only implemented reductions tests for the triangle conditions, we do not provide details on this generalization here. We also note that exploiting this more general condition for effective practical reductions is not straight-forward and seems computationally considerably more expensive than the triangle tests.

\subsection{Further reduction techniques}
\label{sec:reduction:further}

In the following, we propose two additional reduction methods, based on different techniques.
 One uses the reduced-costs of the LP-relaxation of Formulation~\ref{form:dcut},
and one exploits simple symmetries in \mc instances.

We start with the latter. If successful, the test based on the following proposition allows one to contract two (possibly non-adjacent) vertices.
\begin{proposition}
\label{prop:redsym}
Assume there are two distinct vertices $u, v \in V$ such that $N(u) \setminus \{v\} = N(v) \setminus \{u\}$.
 If there exists a non-zero $\alpha$ such that $w(e) = \alpha w(e')$ for all pairs $e,e'$ with $e \in \delta(u) \setminus \{u,v\}, e' \in \delta(v) \setminus \{u,v\}, e \cap e' \neq \emptyset$, and moreover
 \begin{itemize}
 	\item $\{u,v\} \notin E ~\vee~ w(\{u,v\}) < 0$ in case of $\alpha > 0$
 	\item $\{u,v\} \notin E ~\vee~ w(\{u,v\}) > 0$ in case of $\alpha < 0$,
 \end{itemize}
 then there is an optimal vertex solution $y \in \{0,1\}^V$ to $I_{MC}$ such that $y(u) = y(v)$ if $\alpha > 0$, and $y(u) = (1-y(v))$ if $\alpha < 0$.
\end{proposition}
\begin{proof} 
 We consider only the case $\alpha > 0$; the case  $\alpha < 0$ can be shown in a similar way.
  Let $y \in \{0,1\}^V$ with $y(v) \neq y(u)$. We will construct a $y' \in \{0,1\}^V$ with $y'(v) = y'(u)$ such that the weight of the induced cut of $y'$ is not lower than the weight of the induced cut of $y$. In this way, the proof is complete, because we can apply this construction also for any optimal vertex assignment
 
Let $x \in \{0,1\}^E$ be the induced cut of $y$. Assume that 
\begin{equation}
\label{prop:sym:2}
\sum_{e \in \delta(u) \setminus \{u,v\}} w(e) x(e) \geq   \alpha \sum_{e \in \delta(v) \setminus \{u,v\}} w(e) x(e).
\end{equation}
Otherwise, switch the roles of $u$ and $v$ in the following.

Let $f: \delta(v) \setminus \left\{ \{u,v\}  \right\} \to \delta(u) \setminus  \left\{ \{u,v\} \right\}$ such that $e \cap f(e) \neq \emptyset$ for all $e \in  \delta(v) \setminus \{u,v\}$. Note that $f$ is well-defined.
Define a new cut  $x' \in \{0,1\}^E$ as follows
\begin{equation*}
x'(e) := \left\{
\begin{array}{ll}
x(e) & \textrm{if } e \in E \setminus \delta(v) \\
x(f(e) & \textrm{if } e \in \delta(v) \setminus \left\{ \{u,v\}  \right\} \\
0 & \textrm{if } e = \{u,v\} \\
\end{array}
\right. 
\end{equation*}
Because of~\eqref{prop:sym:2} and $\{u,v\} \notin E ~\vee~ w(\{u,v\}) < 0$ it holds that $w^T x' \geq  w^T x$.
\end{proof}
The condition of Proposition~\ref{prop:redsym} can be checked efficiently in practice by using hashing techniques, similar to the ones used for the \emph{parallel rows} test for mixed-integer programs~\cite{Achterberg20}.

A well-known reduction method for binary integer programs, which was already used for \mc~\cite{Barahona89}, is as follows.
Consider a feasible solution $\tilde{x}$ to the LP-relaxation of Formulation~\ref{form:dcut}, with reduced-costs $\tilde{w}$, and with objective value $\tilde{U}$. Further, let $L$ be the weight of a graph cut. If for an $e\in E$ it holds that $\tilde{x}(e) = 0$ and $\tilde{U} -\tilde{w}(e) < L$, one can fix $x(e) := 0$. If for a $e\in E$ it holds that $\tilde{x}(e) = 1$ and $\tilde{U} +\tilde{w}(e) < L$, one can fix $x(e) := 1$. This method can also be used for LP-solutions (obtained during separation) that satisfy only a subset of the cycle inequalities~\eqref{form:dcut:1}.
In the following, we will only consider optimal LP-solutions $\tilde{x}$ (possibly for a subset of the cycle inequalities). Since we furthermore consider only LP-solutions obtained by the Simplex algorithm, all non-zero variables have reduced-cost $0$.

 From incident fixed edges one obtains a (non-unique) partial vertex assignment $y': V' \to \{0,1\}$. This assignment can be used to obtain additional fixings, as detailed in the following proposition.
\begin{proposition}
\label{prop:redcost}
Let $\tilde{x}$ be an optimal solution to the LP-relaxation of Formulation~\ref{form:dcut}, with reduced-costs $\tilde{w}$, and objective value $\tilde{U}$. Let $L$ be an upper bound on the weight of a maximum-cut.
Let $V' \subset V$ and $y': V' \to \{0,1\}$ such that for any optimal vertex assignment $y \in \{0,1\}^V$ it holds that $y(v) = y'(v)$ for all $v \in V'$.
Further, let $u \in V \setminus V'$ and define
\begin{equation*}
\tilde{\Delta}_0 := \sum_{\{u,v\} \in \delta(u) | v \in V',y(v) = 0} \tilde{w}(\{u,v\})
\end{equation*}
and
\begin{equation*}
\tilde{\Delta}_1 := \sum_{\{u,v\} \in \delta(u) | v \in V',y(v) = 1} \tilde{w}(\{u,v\}).
\end{equation*}
For any optimal vertex assignment  $y \in \{0,1\}^V$ the following conditions hold.
If $L + \tilde{\Delta}_0 > \tilde{U}$, then $y(u) = 0$.
If $L + \tilde{\Delta}_1 > \tilde{U}$, then $y(u) = 1$.
\end{proposition}
The proposition follows from standard linear programming results.
If one of the conditions of the proposition is satisfied, one can fix all edges between $u$ and $V'$.

\section{Solving to optimality: branch-and-cut }
\label{sec:branch-and-cut}

This section describes how to incorporate the methods introduced so far together with additional components in an exact branch-and-cut algorithm. 
This branch-and-cut algorithm has been implemented based on the academic MIP solver \scip~\cite{Bestuzheva2021}. Besides
being a stand-alone MIP solver, \scip provides a general branch-and-cut
framework. Most importantly, we rely on \scip for organizing the branch-and-bound search, and the cutting plane management. 
Most native, general-purpose algorithms of \scip such as primal heuristics, conflict analysis, or generic cutting planes are deactivated by our solver for performance reasons.

	\subsection{Key components}
	
	In the following, we list the main components of the branch-and-cut framework that was implemented for this article.
	
	\paragraph{Presolving}
For presolving, the reduction methods described in this article are executed iteratively within a loop. This loop is reiterated as long as at least one edge has been contracted during the previous round, and the predefined maximum number of loop passes has not been reached yet.

	\paragraph{Domain propagation}
For domain propagation we use the reduced-cost criteria described in Section~\ref{sec:reduction:further}. The simple single-edge fixing is done by the generic reduced-costs propagator plug-in of \scip. For the new implication based method we have implemented an additional propagator.

A classic propagation method, e.g.~\cite{Barahona89}, is as follows: Consider the connected components induced by edges that have been fixed to 0 or 1. All additional edges in these connected components can be readily fixed.
However, this technique brought no benefits in our experiments, since the variable values of such edges are implied by the cycle inequalities~\eqref{form:dcut:1}.



	\paragraph{Decomposition}
It is well-known that connected components of the graph underlying a \mc instance can be solved separately, see e.g.~\cite{Juenger2021}.
More generally, one can solve biconnected components separately (this simple observation does not seem to have been mentioned in the \mc literature so far).
Since several benchmark instances used in this article contain many very small biconnected components, we solve components with a limited number of vertices by enumeration. In this way, we avoid the overhead associated with creating and solving a new \mc instance for each subproblem.

	
	\paragraph{Primal heuristics}

Primal heuristics are an important component of practical branch-and-bound algorithms: First, to find an optimal solution (verified by the dual-bound), and second to find strong primal bounds that allow the algorithm to cut off many branch-and-bound nodes.
For computing an initial primal solution, we have implemented the \mc heuristic by Burer et. al.~\cite{Burer02}.
We further use the Kernighan-Lin algorithm~\cite{Kernighan70} to improve any (intermediary) solution found by the algorithm of Burer et. al.
Additionally, we use this combined algorithm as a local search heuristic whenever a new best primal solution has been found during the branch-and-bound search. In this case, we initiate the heuristic with this new best solution (which can be done by translating the solution into the two-dimensional angle vectors required by the heuristic).

We also  implemented the spanning-tree heuristic from Barahona et al.~\cite{Barahona89}, which uses a given (not necessarily optimal) LP-solution to find graph cuts. We execute this heuristic after each separation round.

	\paragraph{Separation}
In each separation round, we initially try to find violated cycle inequalities on triangles of the underlying graph (by enumerating some triangles).
Next, we use shortest-path computations to find additional violated cuts, as described in Section~\ref{sec:relaxation:fast} and Section~\ref{sec:relaxation:post}. Among the speed-up techniques, we have not (yet) implemented the contraction of edges, since
 the separation routine is already quite fast and other implementations seemed more promising.

	\paragraph{Branching}
We simply branch on the edge variables and use the \emph{pseudo-costs} branching strategy of \scip, see Gamrath~\cite{gamrath2020} for more details.
	Initial experiments showed that the default branching strategy of \scip, \emph{reliable pseudo-costs} branching, spends too much time on strong-branching to be competitive.

	\subsection{Parallelization}

 For parallelizing our solver, we use the {\em Ubiquity Generator Framework} (\UG)~\cite{Shinano16}, a software package 
 to parallelize branch-and-bound based solvers---for both shared- and distributed-memory environments.	
 \UG implements a \emph{Supervisor-Worker load coordination scheme}~\cite{Ralphs2018}.  
 Importantly, Supervisor functions make decisions about the load balancing without actually
 	storing the data associated with the branch-and-bound (B\&B) search tree. 

A major problem of parallelizing the B\&B search lies in the simple fact that parallel resources can only be used efficiently once the number of open B\&B nodes is sufficiently large. Thus, we employ so-called \emph{racing ramp-up}~\cite{Ralphs2018}: Initially, each thread (or process) starts the solving process of the given problem instance, but each with different (customized) parameters and random seeds.  Additionally, we reserve some threads to exclusively run primal heuristics. During the racing, information such as improved primal solutions or global variable fixings is being exchanged among the threads.
 We terminate the racing once a predefined number of open B\&B nodes has been created by one thread, or the problem has been solved. Once the racing has been terminated and the problem is still unsolved, the open nodes are distributed among the threads and the actual parallel solving phase starts.

\section{Computational results}
\label{sec:computational}

This section provides computational results on a large collection of \mc and \qubo instances from the literature. We
look at the impact of individual components, and furthermore compare the new solver with the state of the art for the exact solution of  \mc and \qubo instances.
An overview of the test-sets used in the following is given in Table~\ref{tab:instances}. The second column gives the number of instances per test-sets. The third and fourth columns give the range of nodes  and edges in the case of \mc, or the range of variables and non-zero coefficients in the case of \qubo.

Only a few exact \mc or \qubo solvers are publicly available, and some, such as \biqmac~\cite{Rendl10} and \biqbin~\cite{Hrga19}, only as web services.
Still, the state-of-the-art solvers \biqcrunch~\cite{Krislock17} and \madam~\cite{Hrga21} are freely available, even with source code.
However, we have observed that both of these solvers are outperformed on most instances listed in Table~\ref{tab:instances} by the recent release 9.5 of the state-of-the-art commercial solver \gurobi~\cite{gurobi}.
\gurobi solves mixed-integer quadratic programs, which are a superclass of \qubo. In fact, the standard benchmark library for quadratic programs, \qplib~\cite{Furini19}, contains various \qubo instances.
Compared to the previous release 9.1, \gurobi 9.5 has hugely improved on \qubo (and thereby also \mc) instances. For example, while \gurobi 9.1 could not solve any of the \emph{IsingChain} instances from Table~\ref{tab:instances} in one hour (with one thread), \gurobi 9.5 solves all of them in less than one minute.
Thus, in the following, we will use \gurobi 9.5 as a reference for our new solver. 
We will also provide results from the literature, but the comparison with \gurobi 9.5 allows us to obtain results in the same computational environment.


%


Very recently, an article describing a new solver, called \mcsparse, specialized to sparse \mc and \qubo instances was published~\cite{mcsparse}. The computational experiments in~\cite{mcsparse} demonstrate that \mcsparse outperforms previous \mc and \qubo solvers on sparse instances.  
 Like \biqmac~\cite{Rendl10} and \biqbin~\cite{Hrga19}, this solver is only available via a web interface. However, we will still provide some comparison with our solver in the following by using the results published in~\cite{mcsparse}.

The computational experiments were performed on a cluster of Intel Xeon Gold 5122 CPUs with 3.60~GHz, and 96 GB RAM per compute node.
We ran only one job per compute node at a time, to avoid a distortion of the run time measures---originating for example from shared (L3) cache.
For our solver, we use the commerical \cplex 12.10~\cite{cplex} as LP-solver, although our solver also allows for the use of the non-commercial (but slower) \soplex~\cite{Bestuzheva2021} instead.
For \gurobi we set the parameter \emph{MipGap} to $0$. Otherwise, we would obtain suboptimal solutions even for many instances with integer weights.


\begin{table}[ht]
		 	\centering
		 	\scriptsize
		 		\setlength{\tabcolsep}{2pt} 
		 	\begin{tabular*}{1.0\textwidth}{@{\extracolsep{\fill}}lrccl@{}}
		 			 		\toprule
		 		Name &  \# & $|$V$|$ & $|$E$|$ & Description  \\
		 		\midrule
	DIMACS        & 4 &   512 - 3375  &  1536-10125 &  Instances introduced at the\\
		 				 				 															&        &      & & 7th DIMACS Challenge.\\											
		 
		Mannino  & 4 &   48 - 487  &  1128-8511 & Instances from a radio frequency assignment problem,\\
			 				 				 															&        &      & & introduced in~\cite{Bonato2014}.\\		 
	PM1s$_{100}$ & 10 &   100  &  495 &  Instances generated with the \emph{rudy} framework, \\
		 				 				 															&        &      & & 
		 				 				 															from the BiqMac Lib~\cite{wiegele2007biq}. \\		 				 				 															
	W01$_{100}$  & 10 &   100  &  495 &  Instances generated with the \emph{rudy} framework, \\
		 				 				 															&        &      & & from the BiqMac Lib~\cite{wiegele2007biq}.\\	
    K64-chimera   & 80 &   2049  &  8064 &  Instances on D-Wave Chimera graphs\\
						&        &      & & introduced in~\cite{Juenger2021}.\\	
	
	Kernel       & 14 &   33 - 2888  &  91-2981 &  Instances from various sources  \\
				 									&        &      & & collected by~\cite{Ferizovic20}. \\	
	
	IsingChain   & 30 &   100 - 300  &  4950-44850 &  Instances from an application in statistical physics, \\
												&        &      & & introduced in~\cite{Liers2004} \\	
	
	Torus       & 18 &   100 - 343  &  200-1029 &  2D and 3D torus instances from an application in  \\
												&        &      & &  statistical physics, introduced in~\cite{Liers2004}\\
		 				 	\\			 															 				 				 		\midrule	
		 				 	\\													
		Name &  \# & $n$ & $nnz$ & Description  \\
		 		\midrule		 				 				 		 				 				 				
		QPLIB & 22 &   120 - 1225  &  602-34876 &  QUBO instances from the standard benchmark software  \\
			 				 				 															&        &      & & for quadratic programs, see~\cite{Furini19}. \\		 											
			BQP100 & 10 &   100  &  471-528 &  Randomly generated instances \\
		 				 															&        &      & & introduced in ~\cite{beasley1998heuristic}. \\
		 	BQP250 & 10 &   250  & 3039-3208  &  Randomly generated instances \\
		 			 				 															&        &      & & introduced in~\cite{beasley1998heuristic}. \\			 																
	 	BE120.3 & 10 &   120  &  2176-2253 &  Randomly generated instances \\
		 			 				 															&        &      & & introduced in \cite{BillionnetE07}. \\
	 	BE250 & 10 &   250  &  3268-3388 &  Randomly generated instances \\
		 			 				 															&        &      & & introduced in \cite{BillionnetE07}. \\	
		GKA$_{a\text{-}d}$ & 35 &   20 - 125  &  204-7788 &  Randomly generated instances with different densities, \\
													&        &      & & introduced in~\cite{glover1998adaptive} \\		 		
 																			 
		 		\bottomrule
		 	\end{tabular*}
		 	\caption{Details of \mc (upper part) and \qubo (lower part) test-sets used in this article.}
		 	\label{tab:instances}
\end{table}

\subsection{Individual components}

This section takes a look at individual algorithmic components introduced in Section~\ref{sec:relaxation} and Section~\ref{sec:reduction}.

First, we show the run time required for our improved separation of cycle inequalities.
Table~\ref{tab:runparts} reports per test-set the average (arithmetic mean) percentual time required for the separation procedure (column four), as well as for solving the LP-relaxations (column five). Recall that the latter is done by \cplex 12.10, one of the leading commercial LP-solvers. For more than half of the test-sets the average time required for the separation is less than 10 \%. Also for the remaining test-sets it is always less than 20 \%. Notably, this time also includes adding the cuts (including the triangle inequalities) to the cut pool, which requires additional computations.
The time could be further reduced by contracting 0-weight edges in the auxiliary graph, as described in Section~\ref{sec:relaxation}. Notably, both the separation time and LP-solution time are very small for the \emph{IsingChain} and~\emph{Kernel} instances. This behavior is due to the fact that many of these instances are already solved during presolving, as detailed in the following,

	 			\begin{table}[ht]
		 			\centering
		 						\footnotesize
		 			\begin{tabular*}{0.8\textwidth}{@{\extracolsep{\fill}}lrrr@{}}
		 			\toprule
		 			Name &  \# & Sepa-time\,[\%] & LP-time\,[\%] \\
		 			\midrule
BE120.3  & 10 & 2.5 &    79.3 \\		 			
BE250 & 10 &     2.8 &    84.3 \\
BQP100 & 10 &     3.0 &    32.0 \\
BQP250 & 10 &     2.6 &    86.2 \\
GKA$_{a\text{-}d}$ & 35 &     6.7 &    49.1 \\
IsingChain & 30 &     0.0 &     0.0 \\
K64-chimera & 80 &    10.1 &    58.3 \\
Kernel & 14 &     1.3 &     4.8 \\
PM1s$_{100}$ & 10 &    13.1 &    78.3 \\
QPLIB & 22 &    18.1 &    65.5 \\
Torus & 18 &    16.7 &    15.5 \\
W01$_{100}$ & 10 &    12.0 &    59.3 \\				
		 			\hline
		 			\end{tabular*}
		 			\caption{Average times spent in separation and (re-) optimization of the LP for \mc and \qubo test-sets.}\label{tab:runparts}
		 			\end{table}		
	
Next, we demonstrate the strength of the reduction methods implemented for this article. Only results for the \mc test-sets are reported.
 We show the impact of the \mc reduction techniques from~\cite{Lange19} described in Section~\ref{sec:reduction} as well as the \qubo reduction techniques from~\cite{Glover18}---by using the standard problem transformations between \qubo and \mc.
 We refer to the combination of these two as~\emph{base preprocessing}. Additionally, the methods described in Proposition~\ref{prop:redtriag2} and Proposition~\ref{prop:redsym}  are referred to as
 \emph{new techniques}. Note that Proposition~\ref{prop:redcost} cannot be applied, because no reduced-costs are available.

Table~\ref{tab:resultspresolving} shows in the first column the name of the test-set, followed by its number of instances.
The next columns show the percentual average number of nodes and edges of the instances after the preprocessing without (column three and four), and with (columns five and six) the new methods. The last two columns report the percentual relative change between the previous results.
The run time is not reported, because it is on all instances below 0.05 seconds. 

The new reduction techniques have an impact on five of the eight test-sets.
 The strongest reductions occur on \emph{Kernel} and \emph{IsingChain}. We remark that the symmetry-based reductions from Proposition~\ref{prop:redsym} have a very small impact, and only allow for contracting a few dozen additional edges on \emph{Kernel}. We also note that while the \emph{IsingChain} instances are already drastically reduced by the base preprocessing, the new methods still have an important impact, as they reduce the number of edges of several instances from more than a thousand to less than 300.
The \emph{IsingChain} instances were already completely solved by reduction techniques in Tavares~\cite{tavares2008}, by using maximum-flow based methods. However the run time was up to three orders of magnitudes larger than in our case. The machine used by Tavares had a Pentium 4 CPU at 3.60 GHz, thus being significantly slower than the machines used for this article. Still, also when taking the different computing environments into account, the run time difference is huge.

	\begin{table}[ht]
			\centering
		\footnotesize
		\begin{tabular*}{1.0\textwidth}{@{\extracolsep{\fill}}llrrrrrr@{}}
			\toprule
			& &       \multicolumn{2}{c}{base preprocessing} & \multicolumn{2}{c}{+new techniques} & \multicolumn{2}{c}{relative change} \\
			\cmidrule(lr){3-4} \cmidrule(lr){5-6}   \cmidrule(lr){7-8}
			Test-set  & \#  &  $|$V$|$\,[\%] & $|$E$|$\,[\%] &  $|$V$|$\,[\%] & $|$E$|$\,[\%] &  $|$V$|$\,[\%] & $|$E$|$\,[\%]  \\
			\midrule

IsingChain & 30 &  6.1 &     0.8 &   1.1 &     $<$0.05  & \textbf{-82.0} & \textbf{$<$-93.8} \\
K64-chimera & 80 &   3.1 &     4.6 &   3.1 &     4.6  & 0.0 & 0.0 \\
Kernel & 14 &   24.1 &    30.1 &  16.4 &    20.6  & \textbf{-32.0} & \textbf{-31.6} \\
Mannino & 4 & 64.1 &    69.3 &   63.2 &    68.7  & \textbf{-1.4} & \textbf{-0.9} \\
Torus & 18 & 80.6 &    87.5 &    78.5 &    85.2  & \textbf{-2.6} & \textbf{-2.6} \\
W01$_{100}$ & 10 &   99.1 &    94.8 &  99.1 &    94.8  & 0.0 & 0.0 \\
DIMACS & 4 &  97.0 &    98.9 &  96.9 &    98.9  & \textbf{-0.1} & 0.0 \\
PM1s$_{100}$ & 10 &  99.7 &    99.9  &   99.7 &    99.9  & 0.0 & 0.0 \\

			\bottomrule
		\end{tabular*}
		\caption{Average remaining size of \mc instances after preprocessing.}
			\label{tab:resultspresolving}
\end{table}

\subsection{Exact solution}
\label{sec:comp_exact}

 This section compares \gurobi 9.5 and our new solver with respect to the mean time, the maximum time, and the number of solved instances.
For the mean time we use the shifted geometric mean~\cite{Achterberg07a} with a shift of $1$ second. In this section, we use only single-thread mode.
 Table~\ref{tab:exact} provides the results for a time-limit of one hour.
 The second column shows the number of instances in the test-set. Columns three gives the number of instances solved by  \gurobi, column four the number of instances solved by our solver.
  Column five and six show the mean time taken by \gurobi and our solver. The next column gives the relative speedup of our solver.
The last three columns provide the same information for the maximum run time,  Speedups that signify an improved performance of the new solver are marked in bold.

It can be seen that our solver consistently outperforms~\gurobi 9.5---both with respect to mean and maximum time. Also, it solves on each test-set at least as many instances as \gurobi. 
The only test-set where \gurobi performs better is \emph{BQP100}, which, however, can be solved by both solvers in far less than a second.

On the other test-sets, the mean time of the new solver is better, often by large factors (up to $60.07$). On the instance sets that can both be fully solved, the maximum time taken by the new solver is in most cases also much smaller.
 On five of the test-sets, the new solver can solve more instances to optimality than \gurobi 9.5.

\begin{table}[hb]
		\centering
		\footnotesize
		\setlength{\tabcolsep}{2pt}
		\begin{tabular*}{1.0\textwidth}{@{\extracolsep{\fill}}llrrrrrrrr@{}}
			\toprule
			& &     \multicolumn{2}{c}{\# solved} &   \multicolumn{3}{c}{mean time (sh. geo. mean)} & \multicolumn{3}{c}{maximum time} \\
			\cmidrule(lr){3-4} \cmidrule(lr){5-7}   \cmidrule(lr){8-10}
			Test-set  & \#  &   Grb   & new &  Grb\,[s]   & new\,[s] & speedup &~~~    Grb\,[s]   & new\,[s] & speedup  \\
			\midrule
      		    PM1s$_{100}$            &   10  & 10 & 10 &  192.3         &   21.0      &\textbf{   9.16}&  303.3       &  48.6	        &\textbf{   6.24}\\
       		    W01$_{100}$            &   10  & 10 & 10 &  44.1           &  3.1     &\textbf{  14.23}&  97.1      &  21.4		        &\textbf{   4.54}\\
       		    Kernel            &   14  & 14 & 14 &  0.7       &  0.1    &\textbf{   7.00}&  14.3      &  1.1		        &\textbf{  13.00}\\
                 IsingChain            &   30  & 30 & 30 &  1.3      &  $<$0.05    &\textbf{  $>$26.00}&  41.0     &  $<$0.05	   &\textbf{ $>$820.00}\\
                 Torus            &   18  & 18 & 18 &  3.8      &  0.4    &\textbf{   9.50}&  628.0     &   7.6	   &\textbf{  82.63}\\
                 K64-chimera        &   80  & 80 & 80 &  90.1       &  1.5      &\textbf{  60.07}&  195.4     & 8.6		        &\textbf{  22.72}\\
                 QPLIB        &   22  & 8 & \textbf{13} &  687.4     &  173.3     &\textbf{   3.97}&  3600      &   3600		        &   1.00\\
                 BQP100        &   10  & 10 & 10 &  0.1     &  0.1    &   1.00&  0.2      &   0.4		        &   0.50\\
                 BQP250        &   10  & 0 & \textbf{7} &  3600     &  654.1      &\textbf{   5.50}&  3600      &   3600		        &   1.00\\
               BE120.3        & 10 & 9 & \textbf{10} & 265.6 & 60.2 &   \textbf{   4.41} & 3600 & 820.0  & \textbf{   $>$4.40}\\
                 BE250        &   10   & 0 & \textbf{8} &  3600     &  609.3     &\textbf{   5.91}&  3600      &   3600		        &   1.00\\
        GKA$_{a\text{-}d}$    &   35  & 29 & \textbf{31} &  6.5     &  6.0     &\textbf{   1.08}&  3600      &   3600		        &   1.00\\
 
      					\bottomrule
		\end{tabular*}

		\caption{Comparison of \gurobi 9.5 (\emph{Grb}) and new solver (\emph{new}). }  
				\label{tab:exact}
\end{table}

Finally, we compare our solver with the very recent \qubo and \mc solver~\mcsparse, specialized on sparse instances.
In Table~\ref{tab:mcsparse} we provide an instance-wise comparison of our solver and \mcsparse. We provide the number of branch-and-bound nodes (columns three and four) and the run times (columns five and six) of \mcsparse and our solver per problem instance.
We use the 14 instances that were selected in the article by Charfreitag et al.~\cite{mcsparse} as being representative of their test-bed. The first seven instances are \mc and the last seven  \qubo problems.
 Charfreitag et al.~\cite{mcsparse} only use one thread per run. Their results were obtained on a system with AMD EPYC 7543P CPUs at 2.8 GHz, and with 256 GB RAM. CPU benchmarks\footnote{\url{https://www.cpubenchmark.net/singleThread.html\#server-thread}} consider this system to be faster than the one used in this article, already in single-thread mode.
 Furthermore, \mcsparse is embedded into \gurobi (version 9.1), which is widely regarded as the fastest commercial MIP-solver, whereas our solver is based on the non-commercial \scip, although we also use a commercial LP-solver.

 \begin{table}[ht]
 		\centering
 		\footnotesize
 		\setlength{\tabcolsep}{2pt}
 		\begin{tabular*}{1.0\textwidth}{@{\extracolsep{\fill}}lrrrrrr@{}}
 			\toprule
 			& &   & \multicolumn{2}{c}{\# B\&B nodes} & \multicolumn{2}{c}{run time} \\
 			\cmidrule(lr){4-5} \cmidrule(lr){6-7}  
 			Name  & $|$V$|$  &  $|$E$|$ &   MS   & new  &~~~    MS\,[s]   & new\,[s]   \\
 			\midrule
 		pm1s\_100.3      &   100  & 495 & 341     &    737                &     48.2  &  48.0		     \\
 		pw01\_100.0      &   100  & 495 & 171     &    179                &    20.0  &  8.5		     \\
 		mannino\_k487b   &   487  & 5391 & 1     &    15                &    167.3   &  2.9		     \\
 		bio-diseasome   &   516  & 1188 & 1     &    1                &     9.5   &  0.6	     \\
 		ca-netscience   &   379 & 914 & 1     &    1                &     0.1  &  0.0		     \\
 		g000981		    &   110  & 188 & 1     &    1                &     0.0  &  0.0		     \\
 		imgseg\_138032	&   12736 & 23664 & 1     &    1                &    30.5  &  4.4		     \\
 		~\\
 		\midrule
 		~\\
 		Name  & $n$  &  $nnz$ &   MS   & new  &~~~    MS\,[s]   & new\,[s]   \\
 		 			\midrule
       bqp250-3  &   250  & 3092 & 25     &    15                &   414.1   & 85.8		     \\
       gka2c  &   50  & 813 & 1     &    1                &    0.5   &  0.3		     \\
       gka4d  &   100  & 2010 & 129     &    71                &    219.6  &  61.9		     \\
       gka5c  &   80  & 721 & 1     &    1                &    0.1   &  0.1		     \\
       gka7a  &   30  & 241 & 1     &    1                &   0.0   &  0.0		     \\
       be120.3.5  &   120  & 2248 & 111     &    63                &    257.7   &  62.4		     \\
       be250.3  &   250  & 3277 & 107     &    81                &     841.0   &  212.0		     \\
 			\bottomrule
 		\end{tabular*}
 		\caption{Comparison of \mcsparse (\emph{MS}) and our solver (\emph{new}) on seven \mc and seven \qubo instances (considered to be representative~\cite{mcsparse}).}
 		 		\label{tab:mcsparse}  
 \end{table}
 
 As to the number of branch-and-bound nodes, the pictures is somewhat mixed---with~\mcsparse requiring fewer nodes on three, and more nodes on four instances. Notably,~\mcsparse also includes clique-cuts separation, which is not implemented in our solver, and a specialized branching strategy, while we use a simple generic one. These two features might explain the smaller number of nodes on some instances.
As to the run time, five instances can be solved in less than a second by both solvers (with the new solver being slightly faster).
 On the remaining nine instances, the new solver is always faster---for all but one instances by a factor of more than $3$. On one instance (mannino\_k487b), it is even by a factor of more than $50$ faster.

\subsection{Parallelization}

Although parallelization is not the main topic of this article, we still provide some corresponding results in the following. 
To give insights into the strengths and weaknesses of our racing-based parallelization, we provide instance-wise results. We use the test-sets \emph{Mannino} and \emph{DIMACS}, which both contain instances that cannot be solved within one hour by \gurobi and our new solver in single-thread mode. The sizes of the instances are given in Table~\ref{tab:instancesparallel}.

\begin{table}[ht]
		\centering
		\footnotesize
		\setlength{\tabcolsep}{2pt}
		\begin{tabular*}{1.0\textwidth}{@{\extracolsep{\fill}}lrrrrlrr@{}}
			\toprule
			Name  & $|$V$|$  &  $|$E$|$ &  ~~~  & ~~~ &  	Name  & $|$V$|$  &  $|$E$|$  \\
			\midrule
		torusg3-8   & 512 & 1536 & & &       mannino\_k487a.mc   & 487 & 1435 \\	
		toruspm3-8-50   & 512 & 1536 & & &       mannino\_k487b.mc   & 487 & 5391 \\	
		torusg3-15    & 3375 & 10125 & & &       mannino\_k487c.mc   & 487 & 8511 \\	
		toruspm3-15-50   & 3375 & 10125 & & &       mannino\_k48.mc   & 48 &  1128\\								
			\bottomrule
		\end{tabular*}
		\caption{Details on \emph{DIMACS} (left)  and \emph{Mannino} (right) instances.} 
		\label{tab:instancesparallel} 
\end{table}

Table~\ref{tab:resultsparallel} provides results of \gurobi and the new solver on the \emph{DIMACS}  and \emph{Mannino} instances. Both solvers are run once with one thread and once with eight threads. As before, a time-limit of one hour is used.
The table provides the percentual primal-dual gap, as well as the run time. The results reveal for both solvers a performance degradation on easy instances with increased number of threads. Most notably on \emph{mannino\_k487b}, where \gurobi takes almost 10 times longer with eight threads. On the other hand, the new solver shows a strong speedup on two hard instances that cannot be solved in one hour singke-threaded, namely \emph{toruspm3-8-50} and \emph{mannino\_k487c}. On the latter, one even observes a super-linear speedup. This speedup can be at least partly attributed to the exclusive use of primal heuristics on one thread during racing, which finds an optimal solution quickly in both cases. On the other hand, in single-thread mode the best primal solution is sub-optimal even at the time-limit.
 
\begin{table}[ht]
		\centering
		\footnotesize
		\setlength{\tabcolsep}{2pt}
		\begin{tabular*}{1.0\textwidth}{@{\extracolsep{\fill}}lrrrrrrrr@{}}
			\toprule
			&   \multicolumn{4}{c}{primal-dual gap\,[\%]} &   \multicolumn{4}{c}{run time\,[s]} \\
			\cmidrule(lr){2-5} \cmidrule(lr){6-9}
			Name  & Grb-T1  &  Grb-T8 &  new-T1   & new-T8 &   Grb-T1  &  Grb-T8 &  new-T1   & new-T8  \\
			\midrule
		torusg3-8              &   0.0  & 0.0 & 0.0 &  0.0       &  1494.2           &    1178.5                    &     8.5   &  9.3		     \\
		toruspm3-8-50          &   1.8  & 1.8 & 0.5 &  0.0       &  $>$3600          &    $>$3600                   &     $>$3600   &  1415.8		     \\
		torusg3-15             &   6.8  & 3.4 & 1.3 &  0.4       &  $>$3600          &    $>$3600                   &     $>$3600   &  $>$3600	       \\
		toruspm3-15-50         &   9.5  & 12.2 & 2.3 &  2.3       &  $>$3600          &    $>$3600                   &    $>$3600   &  $>$3600	       \\
		mannino\_k487a            &   0.0  & 0.0 & 0.0 &  0.0       &  3.5           &     10.7                &    1.1   &  1.3		     \\
		mannino\_k487b            &   0.0  & 0.0 & 0.0 &  0.0       &  9.2           &     80.5                &    2.9   &  2.8		     \\
		mannino\_k487c             &   0.1   & 0.0 & 0.1 &  0.0    &  $>$3600      & 3176.7               &  $>$3600      &   398.2	       \\
		mannino\_k48             &   0.0  & 0.0 & 0.0 &  0.0      &  0.1          & 0.4             &  2.7                &   3.8	       \\		
			\bottomrule
		\end{tabular*}
		\caption{Results of \gurobi 9.5 (\emph{Grb}) and the new solver (\emph{new}), with one (\emph{-T1}) and eight (\emph{-T8}) threads each.}  
		\label{tab:resultsparallel}
\end{table}

Finally, Table~\ref{tab:best} provides results for several previously unsolved \mc and \qubo benchmark instances from the \emph{QPLIB} and the 7th DIMACS Challenge. We also report the previous best known solution values (\emph{previous primal}) from the literature, which were taken from the \emph{QPLIB} and the \emph{MQLib}~\cite{Dunning18}.
		For the \emph{QPLIB} instances we report the results from the one hour single-thread run in Section~\ref{sec:comp_exact}. However, for the DIMACS instances, \emph{torusg3-15} and \emph{toruspm3-15-50},
		we performed additional runs. 	Note that the DIMACS instances were originally intended to be solved with negated weights. However, it seems that most publications, e.g.,~\cite{Dunning18}, do not perform this transformation. Thus, we also use the unmodified instances, to allow for better comparison. However, we additionally report the solution values of the transformed instances, these transformed instances are marked by a $\star$.
		We used a machine with 88 cores of Intel Xeon E7-8880 v4 CPUs @ 2.20GHz. We ran the two instances (non-exclusively) for at most 3 days while using 80 threads. Both \emph{torusg3-15} and \emph{torusg3-15$^\star$} could be solved to optimality in this way, but \emph{toruspm3-15-50} and \emph{toruspm3-15-50$^\star$} still remained with a primal-dual gap of $1.8$ percent each.

	 			\begin{table}[ht]
		 			\centering
		 						\footnotesize
		 			\begin{tabular*}{0.8\textwidth}{@{\extracolsep{\fill}}lrrr@{}}
		 			\hline
		 			Name &  gap [\%] & new primal & previous primal \\
		 			\hline
		 		torusg3-15 & \textbf{opt} &  \textbf{286626481} & 282534518 \\
		 		torusg3-15$^\star$ & \textbf{opt} &  \textbf{292031950} & - \\
		 		toruspm3-15-50    & 1.8 &     {3010} & 2968 \\
		 		toruspm3-15-50$^\star$    & 1.8 &     {3008} & - \\
		 	\midrule
		 	QPLIB\_3693   & 1.3  & -1152    &                     -1148 \\
        	QPLIB\_3850 &   1.7 &    -1194   &                  -1192			 	\\								 				
		 			\hline
		 			\end{tabular*}
		 			\caption{Improved solutions for \mc (first four) and \qubo (last two) benchmark instances.}\label{tab:best}
		 			\end{table}

\section{Conclusion and outlook}

This article has demonstrated how to design a state-of-the-art solver for sparse \qubo and \mc instances, by enhancing and combining key algorithmic ingredients such as presolving and cutting-plane generation. The newly implemented solver outperforms both the leading commercial and non-commercial competitors on a wide range of test-sets from the literature.
Moreover, the best known solutions to several instances could be improved.

Still, there are various promising routes for further improvement. Examples would be a new branching strategy, or the implementation of additional separation methods such as clique-cuts~\cite{Bonato2014}.
In this way, a considerable further speedup of the new solver might be achieved.

			\bibliographystyle{plain}
			\bibliography{main.bib}

\end{document}